\documentclass[12pt]{amsart}
\usepackage{amsmath,amssymb,amsthm, verbatim,hyperref}
\usepackage{fullpage}
\usepackage{xcolor}
\usepackage{url}
\usepackage{mathrsfs}
\usepackage[all]{xy}
  \SelectTips{cm}{10}
  \everyxy={<2.5em,0em>:}
\usepackage{tikz}
\usepackage{picinpar} 

\DeclareMathOperator{\dist}{dist}

\DeclareMathOperator{\exc}{Exc}
\DeclareMathOperator{\mult}{mult}
\DeclareMathOperator{\Star}{Star}

\usepackage{enumerate, amsmath, amsthm, amsfonts, amssymb,  mathrsfs}

\theoremstyle{plain}
  \newtheorem{lemma}[equation]{Lemma}
  \newtheorem{proposition}[equation]{Proposition}
  \newtheorem{theorem}[equation]{Theorem}
  \newtheorem{corollary}[equation]{Corollary}

    \newtheorem{conjecture}[equation]{Conjecture}

\theoremstyle{definition}
  \newtheorem{definition}[equation]{Definition}

\theoremstyle{remark}
  \newtheorem{remark}[equation]{Remark}

\renewcommand{\thesection}{\arabic{section}}
\renewcommand{\theequation}{\thesection.\arabic{equation}}

 \DeclareFontFamily{U}{manual}{}
 \DeclareFontShape{U}{manual}{m}{n}{ <->  manfnt }{}
 \newcommand{\manfntsymbol}[1]{%
    {\fontencoding{U}\fontfamily{manual}\selectfont\symbol{#1}}}

\makeatletter
   \@addtoreset{section}{part}
   \@addtoreset{equation}{section}
   \@addtoreset{footnote}{section}

    {\hspace*{\fill}$\lrcorner$\endgraf\endgroup\end{trivlist}}
 
\makeatother

  \DeclareFontFamily{OT1}{pzc}{}
  \DeclareFontShape{OT1}{pzc}{m}{it}{<-> s * [1.100] pzcmi7t}{}
  \DeclareMathAlphabet{\mathpzc}{OT1}{pzc}{m}{it}

\newif\ifhascomments \hascommentstrue
\ifhascomments
  \newcommand{\david}[1]{{\color{red}[[\ensuremath{\bigstar\bigstar\bigstar} #1]]}}
  \newcommand{\matt}[1]{{\color{red}[[\ensuremath{\spadesuit\spadesuit\spadesuit} #1]]}}
\else
  \newcommand{\david}[1]{}
  \newcommand{\matt}[1]{}
\fi

\newcommand{\<}{\langle}
\renewcommand{\>}{\rangle} 

\newcommand{\A}{\mathcal A}
\renewcommand{\AA}{\mathbb{A}}

\newcommand{\bbar}[1]{\overline{#1}}

\newcommand{\C}{\mathcal C}

\DeclareMathOperator{\Endo}{\ensuremath{\mathcal{E}\kern-.125em\mathpzc{nd}}}

\newcommand{\F}{\mathcal F}

\newcommand{\G}{\mathcal G}

\newcommand{\GG}{\mathbb G}

\DeclareMathOperator{\Hom}{\ensuremath{\mathcal{H}\kern-.125em\mathpzc{om}}}
\newcommand{\id}{\mathrm{id}}

\newcommand{\m}{\mathfrak m}
\newcommand{\M}{\mathcal M}
\newcommand{\N}{\mathcal N}

\renewcommand{\O}{\mathcal O}

\newcommand{\PP}{\mathbb{P}}

\newcommand{\QQ}{\mathbb Q}
\newcommand{\RR}{\mathbb R}

\renewcommand{\setminus}{\smallsetminus}

\newcommand{\Z}{\mathcal{Z}}
\newcommand{\ZZ}{\mathbb{Z}}

 \def\ari[#1]{\ar@{^(->}[#1]}
 \def\are[#1]{\ar[#1]^{\txt{\'et}}}
 \def\areh[#1]{\ar[#1]|{\txt{$H$-eq}}^{\txt{\'et}}}
 \def\ars[#1]{\ar@{->>}[#1]}
 \newcommand{\dplus}{\ar@{}[d]|{\mbox{$\oplus$}}}
 \newcommand{\dtimes}{\ar@{}[d]|{\mbox{$\times$}}}

\usepackage{amsthm}

\theoremstyle{plain}

\newtheoremstyle{named}{}{}{\itshape}{}{\bfseries}{.}{.5em}{\thmnote{#3 }#1}
\theoremstyle{named}

\usepackage{amsmath}

\DeclareMathOperator{\codim}{codim}

\newcommand{\NE}{\bbar{\nee}}
\DeclareMathOperator{\nee}{NE}
\DeclareMathOperator{\Nef}{Nef}

\begin{document}
\newtheorem{thm}{Theorem}[section]
\newtheorem{lem}[thm]{Lemma}
\newtheorem{dfn}[thm]{Definition}
\newtheorem{cor}[thm]{Corollary}
\newtheorem{conj}[thm]{Conjecture}
\newtheorem{clm}[thm]{Claim}
\theoremstyle{remark}
\newtheorem{exm}[thm]{Example}
\newtheorem{rem}[thm]{Remark}
\newtheorem{que}[thm]{Question}
\def\N{{\mathbb N}}
\def\G{{\mathbb G}}
\def\F{{\mathbb F}}
\def\Q{{\mathbb Q}}
\def\R{{\mathbb R}}
\def\C{{\mathbb C}}
\def\P{{\mathbb P}}
\def\A{{\mathbb A}}
\def\Z{{\mathbb Z}}
\def\v{{\mathbf v}}
\def\w{{\mathbf w}}
\def\x{{\mathbf x}}
\def\O{{\mathcal O}}
\def\M{{\mathcal M}}
\def\kbar{{\bar{k}}}
\def\tr{\mbox{Tr}}
\def\id{\mbox{id}}

\renewcommand{\theenumi}{\alph{enumi}}

\title{Approximating rational points on toric varieties}

\author{David McKinnon}
\address{University of Waterloo \\
Department of Pure Mathematics \\
Waterloo, Ontario \\
Canada  N2L 3G1}
\email{dmckinnon@uwaterloo.ca}

\author{Matthew Satriano}
\address{University of Waterloo \\
Department of Pure Mathematics \\
Waterloo, Ontario \\
Canada  N2L 3G1}
\email{msatrian@uwaterloo.ca}

\thanks{The authors were partially supported by Discovery Grants from the Natural Sciences and Engineering Research Council.}

\begin{abstract}
Given a smooth projective variety $X$ over a number field $k$ and $P\in X(k)$, the first author conjectured that in a precise sense, any sequence that approximates $P$ sufficiently well must lie on a rational curve. We prove this conjecture for smooth split toric surfaces conditional on Vojta's conjecture. More generally, we show that if $X$ is a $\QQ$-factorial terminal split toric variety of arbitrary dimension, then $P$ is better approximated by points on a rational curve than by any Zariski dense sequence.
\end{abstract}

\maketitle
\tableofcontents

\section{Introduction}
\label{sec:intro}

In Dirichlet's 1842 Approximation Theorem, he showed that for every irrational number $x$, there exist infinitely many rational numbers $\frac{a}{b}$ in reduced form satisfying the equation $|x-\frac{a}{b}|<\frac{1}{b^2}$. His result can be rephrased as follows. For a point $x\in \RR$ the {\em approximation exponent} $\tau_{x}$ of $x$ is the unique extended real number $\tau_{x}\in(0,\infty]$ such that the inequality
$$\left|{x-\frac{a}{b}}\right| \leq \frac{1}{b^{\tau_{x}+\delta}}$$
has only finitely many solutions $\frac{a}{b}\in \QQ$ in reduced form whenever $\delta>0$, and has infinitely solutions whenever $\delta<0$. The approximation exponent measures a certain tension between our ability to closely approximate $x$ by rational numbers (the distance term $|x-\frac{a}{b}|$)  and the complexity (the $\frac{1}{b}$ term) of the number required to make this approximation. In this notation, Dirichlet's theorem then states $\tau_x\geq2$ for irrational $x$. In 1844, Liouville \cite{L} proved that if $x\in \RR$ is algebraic of degree $d$ over $\QQ$, then $\tau_{x}\leq d$. This upper bound was subsequently improved by Thue \cite{thue-approx} in 1909, Siegel \cite{siegel-approx} in 1921, and independently by Dyson \cite{dyson-approx} and Gelfand in 1947, leading finally to Roth's famous 1955 theorem \cite{roth-rat-approx} that $\tau_x\leq2$ for all algebraic $x\in\RR$. Therefore, Dirichlet's Theorem and Roth's Theorem together show that $\tau_x=2$ for all irrational $x$.

McKinnon and Roth \cite{MR2} generalized $\tau_x$ to arbitrary projective varieties $X$ over a number field $k$ by replacing the function $|x-\frac{a}{b}|$ by a distance function $\dist_v(x,\cdot)$ depending on a place $v$ of $k$, and measuring the complexity of a point via a height function $H_D(\cdot)$ depending on an ample divisor $D$. An essential change, however, is that they moved the exponent $\tau_x$ from the height to the distance; this was done to make their generalized exponents behave better with respect to changes in $D$. Given any sequence $\{x_i\}$ approximating $x$, one then obtains an associated \emph{approximation constant} $\alpha_{x,\{x_i\}}(D)$, see Section \ref{sec:review-of-alpha} for the precise definition. The constant $\alpha_x(D)$ is defined to be the infimum of $\alpha_{x,\{x_i\}}(D)$ over all choices of sequences $\{x_i\}$; if one restricts attention only to sequences contained in a subvariety $Z\subseteq X$, then the resulting infimum is denoted by $\alpha_{x,Z}(D)$.

The focus of our paper is a conjecture introduced by the first author in 2007:

%
%
%

\begin{conjecture}[{\cite[Conjecture 2.7]{McK}}]
\label{conj:ratcurve}
Let $X$ be an algebraic variety defined over a number field $k$, and $D$ any ample divisor on $X$.  Let $P\in X(k)$ and assume that there is a rational curve defined over $k$ passing through $P$. Then there exists a curve $C\subseteq X$ (necessarily rational) for which $\alpha_{P,C}(D)=\alpha_P(D)$.
\end{conjecture}

This conjecture is known in some special cases, primarily in dimension $2$: it was shown for split rational surfaces of Picard rank at most four in \cite{McK}, cubic surfaces in \cite{MR}, and blow-ups of the $n$-th Hirzebruch surface at special configurations of at most $2n$ points in \cite{C}. The conjecture was also verified in \cite{generic-tv} for smooth projective split toric varieties $X$ with torus $T$ when $P\in T(k)$ and the pseudo-effective cone $\overline{\mathrm{Eff}}(X)$ is simplicial. Unfortunately, this is a rather restrictive condition: it is equivalent to the combinatorial hypothesis that there exists a maximal cone $\sigma$ in the fan of $X$ such that every ray outside $\sigma$ is a negative linear combination of the rays of $\sigma$, see \cite[Lemma 6.2]{generic-tv}. In particular, all of the aforementioned results still leave open the case of smooth split toric surfaces even if one requires $P\in T(k)$.

In this work, we considerably extend the list of cases where Conjecture \ref{conj:ratcurve} is known: we prove it not only for all smooth split toric surfaces $X$ and arbitrary $P\in X(k)$ conditional on Vojta's Conjecture, but we also obtain approximation results more generally for $\QQ$-factorial terminal singularities on projective split toric varieties of arbitrary dimension.

%

The starting point for our work is a new class of points that we now introduce.

\begin{definition}
\label{def:property*}
Let $X$ be a $\QQ$-Gorenstein algebraic variety defined over a number field $k$. We say $X$ is \emph{canonically bounded at} $P\in X(k)$ if $\alpha_{P,\{x_i\}}(-K_X)\geq \dim X$ for all Zariski dense sequences $\{x_i\}$.
\end{definition}

Canonical boundedness is a highly natural notion. Indeed, we show that every point on a smooth variety is conjecturally canonically bounded:

\begin{proposition}
\label{prop:vojta->canonically-bounded}
Let $X$ be a smooth projective variety over a number field $k$. Then Vojta's Main Conjecture implies that $X$ is canonically bounded at every point $P\in X(k)$.
\end{proposition}

Our first main result is that Conjecture \ref{conj:ratcurve} holds for split toric surfaces in the presence of the canonical boundedness condition:

\begin{theorem}
\label{thm:a2conj-surface}
Let $X$ be a split toric surface over a number field $k$ and let $P\in X(k)$ be a smooth point that is canonically bounded in the minimal resolution of $X$. Then Conjecture~\ref{conj:ratcurve} holds at $P$ for every nef divisor $D$ on $X$.
\end{theorem}

In fact, Theorem \ref{thm:a2conj-surface} follows from a much more general theorem which we prove for all higher dimensional split toric varieties. Given a split toric variety $X$ over a number field $k$, we say $f\colon\widetilde{X}\to X$ is a \emph{terminal resolution} if it is a proper birational toric morphism defined over $k$ and $\widetilde{X}$ is $\QQ$-factorial, projective, and has at worst terminal singularities.

\begin{theorem}
\label{thm:conj-higher-dim}
Let $X$ be a split toric variety over a number field $k$ and let $P\in X(k)$. Suppose $f\colon\widetilde{X}\to X$ is a terminal resolution which is an isomorphism at $P$, and that $P$ is canonically bounded in $\widetilde{X}$.

Then for all $\QQ$-Cartier nef divisors $D$ on $X$, there exists an irreducible rational curve $C$ through $P$ such that $C$ is unibranch at $P$ and
\[
\alpha_{P,C}(D)\leq\alpha_{P,\{x_i\}}(D)
\]
for all Zariski dense sequences $\{x_i\}$.
\end{theorem}

\begin{remark}
\label{rmk:higher-dimensional-conj}
Theorem \ref{thm:conj-higher-dim} says there exists a curve $C$ whose $\alpha$ value is smaller than that of every Zariski dense sequence.  Notice that this does not imply Conjecture \ref{conj:ratcurve} in higher dimensions since it is possible that there exists a subvariety $Z$ with $1<\dim Z<\dim X$ for which $\alpha_{P,Z}(D)<\alpha_{P,C}(D)$.  However, when $X$ is a surface, no such $Z$ can exist. Hence, Theorem \ref{thm:conj-higher-dim} implies Conjecture \ref{conj:ratcurve} for surfaces, i.e.~Theorem \ref{thm:conj-higher-dim} implies Theorem \ref{thm:a2conj-surface}.
\end{remark}

\begin{remark}
\label{rmk:not-best-curve}
A subtle point here is that the curve $C$ we construct in the proof of Theorem~\ref{thm:conj-higher-dim} need not satisfy $\alpha_{P,C}(D)=\alpha_P(D)$, even for surfaces. That is, we show $\alpha_{P,C}(D)\leq\alpha_{P,\{x_i\}}(D)$ for all Zariski dense sequences $\{x_i\}$, and for surfaces, this is enough to guarantee the existence of some auxiliary curve $C'$ with $\alpha_{P,C'}(D)=\alpha_P(D)$, but $C'$ may not equal $C$. Indeed, our construction of $C$ is independent of the number field $k$, but in Section~\ref{sec:subtlety} we show that for $P=[1:1:1]\in\PP(4,7,13)$, the value of $\alpha_P(D)$ depends on $k$. In particular, any proof of Theorem \ref{thm:conj-higher-dim} without assuming \emph{a priori} that $P$ is canonically bounded must include an explanation for the subtle fact that certain curves such as $C'$ are contained in the Zariski closed locus of exceptions to the canonical boundedness condition provided by Vojta's Conjecture.
\end{remark}

Finally, combining Proposition \ref{prop:vojta->canonically-bounded} and Theorem \ref{thm:conj-higher-dim} yields:

\begin{theorem}
\label{thm:a2conj}
Let $X$ be a split toric variety over a number field $k$ and assume Vojta's Main Conjecture holds for some projective toric (strong) resolution of singularities of $X$. Then for all smooth points $P\in X(k)$ and all $\QQ$-Cartier nef divisors $D$ on $X$, there exists an irreducible rational curve $C$ through $P$ such that $C$ is unibranch at $P$ and
\[
\alpha_{P,C}(D)\leq\alpha_{P,\{x_i\}}(D)
\]
for all Zariski dense sequences $\{x_i\}$.
\end{theorem}

\addtocontents{toc}{\protect\setcounter{tocdepth}{0}}
\section*{Acknowledgments}
\addtocontents{toc}{\protect\setcounter{tocdepth}{1}}
It is a pleasure to thank Anton Geraschenko, Fei Hu, Brian Lehmann, John Lesieutre, Mircea Musta\textcommabelow{t}\u{a}, Mike Roth, and Karl Schwede for helpful discussions.

\section{Key properties of the approximation constant $\alpha_P$}
\label{sec:review-of-alpha}

In this section, we collect the relevant facts we need about the approximation constant. For a more detailed discussion of $\alpha$, see \cite{MR2}. Proofs of all of the facts below can be found in \cite{MR}.


%

\begin{definition}\label{seqappconst}
	Let $X$ be a projective variety over a number field $k$, let $P\in X(\overline{k})$, and let $D$ be a $\QQ$-Cartier divisor 
	on $X$.  For any sequence $\{x_i\}\subset X(k)$ of distinct
	points with $\dist_v(P,x_i)\rightarrow 0$, which we denote by $\{x_i\}\to P$, we set
	$$A(\{x_i\}, D) = \left\{{
		\gamma\in\RR\mid
		\dist_v(P,x_i)^{\gamma} H_D(x_i)\,\,\mbox{is bounded from above}
	}\right\}.
	$$
\end{definition}

\begin{remark}
It follows immediately from the definition that if $A(\{x_i\}, D)$ is nonempty then it is an
interval unbounded to the right, i.e., if $\gamma\in A(\{x_i\},D)$ then $\gamma+\delta\in A(\{x_i\},D)$ for any $\delta>0$.
\end{remark}

\begin{definition}
	With hypotheses as in Definition \ref{seqappconst}, if $A(\{x_i\},D)$ is empty we set $\alpha_{P,\{x_i\}}(D)=\infty$.  Otherwise we set
	$\alpha_{P,\{x_i\}}(D)$ to be the infimum of $A(\{x_i\},D)$.  We call $\alpha_{P,\{x_i\}}(D)$ the \emph{approximation constant}
	of $\{x_i\}$ with respect to $D$.
\end{definition}

\begin{remark}
If $\{x_i'\}$ is a subsequence of $\{x_i\}$ then $A(\{x_i\},D)\subseteq A(\{x_i'\},D)$.  In particular, $\alpha_{P,\{x_i'\}}(D)\leq \alpha_{P,\{x_i\}}(D)$, so we may freely replace a sequence with a subsequence when trying to establish lower bounds.
\end{remark}

As $i\to\infty$ we have $\dist_v(P,x_i)\to0$.  We thus expect that $\dist_v(P,x_i)^{\gamma}H_{D}(x_i)$ goes to $0$ for large $\gamma$ and to $\infty$ for small $\gamma$.  The number $\alpha_{P,\{x_i\}}(D)$ marks the transition point between these two behaviours.

\begin{definition}\label{def:alpha}
	Let $k$ be a number field, $X$ a projective variety over $k$, $D$ a $\QQ$-Cartier divisor on $X$, and $P\in X(\overline{k})$.  Then $\alpha_{P}(D)$ is defined to be the infimum of all $\alpha_{P,\{x_i\}}(D)$ as we range over sequences of distinct points $\{x_i\}\subset X(k)$ converging to $P$.  If no such sequence exists then set $\alpha_{P}(D)=\infty$.
\end{definition}

To expand upon the connection between $\alpha_x$ and the usual approximation exponent $\tau_x$ as defined in the Introduction, suppose that $D$ is an ample divisor on $X$.  We may define an approximation constant $\tau_P(D)$ by simply extending the definition on $\PP^1$, namely by defining $\tau_P(D)$ to be the unique extended real number $\tau_P(D)\in [0,\infty]$ such that the inequality
$$\dist_v(P,Q) < \frac{1}{H_{D}(Q)^{\tau_P(D)+\delta}}$$
has only finitely many solutions $Q\in X(k)$ whenever $\delta>0$ and has infinitely many solutions $Q\in X(k)$ whenever $\delta<0$. Then \cite[Proposition 2.11]{MR} implies that $\alpha_P(D) = \frac{1}{\tau_P(D)}$. In particular the theorem of Liouville becomes $\alpha_{P}(\mathcal{O}_{\PP^1}(1))\geq \frac{1}{d}$ for $P\in \RR$ of degree $d$ over $\QQ$, and it is this type of lower bound that we wish to generalize to arbitrary varieties. We use the reciprocal of $\tau$ because $\alpha$ behaves more naturally when we vary $D$ (see, for example, Proposition~2.9 of \cite{MR} for more details).

We will need one further property of $\alpha_P$. Theorem~2.6 of \cite{MR} states the following, see also Theorem~2.16 of \cite{MR2}:

\begin{theorem}\label{thm:curve}
	Let $C$ be an irreducible $k$-rational curve and $\varphi\colon\PP^1\rightarrow C$ its normalization map.
	Then for any ample divisor $D$ on $C$, and any $P\in C(\overline{k})$ we have the equality:
	\[\alpha_{P,C}(D)=\min_{Q\in \varphi^{-1}(P)} \frac{d}{r_{Q} m_{Q}}\]
	where $d=\deg(D)$, $m_{Q}$ is the multiplicity of the branch of $C$ through $Q$ corresponding to $Q$, and
	\[r_{Q}=
	\begin{cases}
	0 & \text{if $\kappa(Q)\not\subseteq k_v$} \\
	1 &\text{if $\kappa(Q)=k$} \\
	2 &\text{otherwise.}
	\end{cases}
	\]
\end{theorem}
We are primarily interested in the case where the curve $C$ is unibranch at $P$, so there is only one point $Q\in\varphi^{-1}(P)$ which necessarily has $r_Q=1$. Thus, we have the following result.

\begin{theorem}\label{thm:unibranch-curve}
Let $X$ be a variety defined over a number field $k$, and let $C$ be an irreducible rational curve on $X$, with $C$ also defined over $k$.  Let $P$ be a $k$-rational, unibranch point of $C$, and let $D$ be a $\QQ$-Cartier divisor on $X$.  Then 
\[\alpha_{P,C}(D)=\frac{1}{m}C\cdot D\]
where $m$ is the multiplicity of $P$ on $C$.
\end{theorem}

\section{Vojta's Main Conjecture and canonical boundedness}
\label{sec:prelim-results}

The goal in this section is to show that Vojta's Main Conjecture implies every point of a smooth projective variety is canonically bounded, i.e.~we prove Proposition \ref{prop:vojta->canonically-bounded}. We turn to the proof after recalling for the reader's convenience the statement of the conjecture \cite{Vo}.

\begin{conj}[Vojta's Main]\label{conj:vstrong}
Let $X$ be a smooth algebraic variety defined over a number field $k$, with canonical divisor $K$.  Let $S$ be a finite set of places of $k$.  Let $A$ be a big divisor on $X$, and let $D$ be a normal crossings divisor on $X$.  Choose height functions $h_K$ and $h_A$ for $K$ and $A$, respectively, and define a proximity function $m_S(D,P) = \sum_{v\in S} h_{D,v}(P)$ for $D$ with respect to $S$, where $h_{D,v}$ is a local height function for $D$ at $v$.  Choose any $\epsilon>0$.  Then there exists a nonempty Zariski open set $U=U(\epsilon)\subseteq X$ such that for every $k$-rational point $Q\in U(k)$, we have the following inequality:
\begin{equation}\label{vojta-ineq}
m_S(D,Q) + h_K(Q) \leq \epsilon h_A(Q).
\end{equation}
\end{conj}

We now turn to Proposition \ref{prop:vojta->canonically-bounded}.

\begin{proof}[{Proof of Proposition \ref{prop:vojta->canonically-bounded}}]
Let $\dim X=n$ and fix a place $v$ of $k$. Let $S=\{v\}$, $A$ be any very ample divisor on $X$, and $D$ be the union of any $n$ normal crossings divisors that intersect properly and transversely at $P$. Then by our choice of $D$, we have $m_S(D,Q)\geq -n\log\dist_v(P,Q)$.  Fix any $\epsilon>0$. If $Q$ satisfies inequality (\ref{vojta-ineq}), then
\begin{equation}\label{can-bnd-ineq}
\dist_v(P,Q)^nH_{-K_X}(Q) \geq H_A(Q)^{-\epsilon}.
\end{equation}

Since $A$ is very ample, we have $\alpha_P(A)\geq 1$ by \cite[Proposition~2.15.(d)]{MR2}.  So, by definition of $\alpha_{P}(A)$, for any $k$-rational point $Q$ on $X$, we have 
\begin{equation}\label{ample-ineq}
\dist_v(P,Q)^\epsilon H_A(Q)^\epsilon\geq \kappa
\end{equation}
for some positive constant $\kappa$ depending on $\epsilon$ but not $Q$. Therefore if $Q$ satisfies inequality (\ref{vojta-ineq}), then combining inequalities (\ref{can-bnd-ineq}) and (\ref{ample-ineq}), we deduce
\[\dist_v(P,Q)^{n-\epsilon}H_{-K_X}(Q)\geq \kappa.\]
In particular, if $\{x_i\}$ is a sequence satisfying $\alpha_{P,\{x_i\}}(-K_X)<n$, then choosing $\epsilon$ sufficiently small, we see $\{x_i\}$ must be eventually contained in the complement of the set $U(\epsilon)$ from Vojta's Main Conjecture.  So, $\{x_i\}$ must be contained in a finite union of proper subvarieties, as desired.
\end{proof}

\section{Preliminary reductions in the proof of Theorem \ref{thm:conj-higher-dim}}
For the remainder of the paper, we fix a number field $k$ a place $v$ of $k$, and a $v$-adic distance function $\dist_v$ which we will denote by $\dist$. We begin by reducing Theorem \ref{thm:conj-higher-dim} to the case where $X$ is $\QQ$-factorial with terminal singularities itself.

\begin{proposition}
\label{prop:reduce-to-sm-or-terminal-case}
Let $X$ be a split toric variety defined over a number field $k$, let $P\in X(k)$, and let $D$ be a $\QQ$-Cartier nef divisor on $X$. Suppose $f\colon\widetilde{X}\to X$ is a toric proper birational map which is an isomorphism at $P$, and that there is an irreducible rational curve $C\subseteq\widetilde{X}$ through $f^{-1}(P)$ such that $C$ is unibranch at $f^{-1}(P)$ and
\[
\alpha_{f^{-1}(P),C}(f^*D)\leq\alpha_{f^{-1}(P),\{x_i\}}(f^*D)
\]
for all Zariski dense sequences $\{x_i\}$.  Then the curve $f(C)$ is an irreducible rational curve that is unibranch at $P$ and satisfies
\[
\alpha_{P,f(C)}(D)\leq\alpha_{P,\{x_i\}}(D).
\]
for all Zariski dense sequences $\{x_i\}$.
\end{proposition}
\begin{proof}
Irreducibility of $f(C)$ follows from that of $C$. Moreover, since $f$ is an isomorphism at $P$, the fact that $C$ is rational and unibranch at $f^{-1}(P)$ immediately implies that $f(C)$ is rational and unibranch at $P$. Lastly, applying Corollary~8.6 of \cite{MR2} to the subset of $X$ on which $f$ is an isomorphism implies that $\alpha_{P,f(C)}(D)\leq\alpha_{P,\{x_i\}}(D)$ for all Zariski dense sequences $\{x_i\}$.
\end{proof}

By Proposition~\ref{prop:reduce-to-sm-or-terminal-case}, to prove Theorem \ref{thm:conj-higher-dim}, we can assume that our split toric variety $X$ is projective, $\QQ$-factorial, and has at worst terminal singularities. Thus, it remains to prove the following theorem, which is slightly more general.

\begin{theorem}
\label{thm:conj-plus-K-degree}
Let $X$ be a projective terminal $\QQ$-factorial split toric variety over a number field $k$. Let $P\in X(k)$ and $D$ be a nef divisor on $X$. If $P$ is canonically bounded, then there exists an irreducible curve $C$ through $P$ which is unibranch at $P$ and $$\alpha_{P,C}(D)\leq \alpha_{P,\{x_i\}}(D)$$ for all Zariski dense sequences $\{x_i\}$ on $X$. Moreover, if $X\not\simeq\PP^n$, then we can choose $C$ so that $-K_X\cdot C\leq\dim X$.
\end{theorem}

We prove Theorem \ref{thm:conj-plus-K-degree} using an induction argument via the Minimal Model Program (MMP). In order to explain this, we begin with several preliminary results. 

\begin{lemma}
\label{l:reduce-to-a=0}
Let $X$ be a $\QQ$-factorial algebraic variety over a number field $k$ which is canonically bounded at $P\in X(k)$, Let $a\geq0$ and $D$ be a nef divisor on $X$ such that $D+aK_X$ is also nef. Suppose $C$ is an irreducible rational curve through $P$ which is unibranch at $P$, $-K_X\cdot C\leq\dim X$, and $$\alpha_{P,C}(D+aK_X)\leq \alpha_{P,\{x_i\}}(D+aK_X)$$ for all Zariski dense sequences $\{x_i\}$ on $X$. Then for all Zariski dense sequences $\{x_i\}$ on $X$, we have $$\alpha_{P,C}(D)\leq \alpha_{P,\{x_i\}}(D)$$ as well.
\end{lemma}
\begin{proof}
Since $C$ is unibranch at $P$, Theorem~\ref{thm:unibranch-curve} gives us that $\alpha_{P,C}(F)=\frac{1}{m}C\cdot F$ for every nef divisor $F$, where $m$ is the multiplicity of $C$ at $P$.   In particular, 
\[
\alpha_{P,C}(D)=\frac{1}{m}C\cdot D=\frac{1}{m}C\cdot(D+aK_X)-\frac{a}{m}K_X\cdot C\leq \alpha_{P,C}(D+aK_X)+a\dim X.
\]
Using the defining property of $C$ and the fact that $X$ is canonically bounded at $P$, we see
\[
\alpha_{P,C}(D)\leq \alpha_{P,\{x_i\}}(D+aK_X)+a\alpha_{P,\{x_i\}}(-K_X).
\]
Lastly, concavity of $\alpha$, shown in \cite[Proposition~2.14.(b)]{MR2}, yields
\[
\alpha_{P,C}(D)\leq \alpha_{P,\{x_i\}}(D),
\]
proving the desired result for $D$.
\end{proof}

Now, let $X$ be a projective $\QQ$-factorial split toric variety over a number field $k$ which is canonically bounded at $P\in X(k)$, and let $D$ be a nef divisor on $X$. Since $X$ is toric, the Mori cone $\NE(X)$ is polyhedral. Let $C_0,\dots,C_\ell$ be the torus-invariant curves generating the $K_X$-negative extremal rays, and set
\begin{equation}
\label{eqn:a-move-to-boundary}
a=\min_i\frac{D\cdot C_i}{-K_X\cdot C_i};
\end{equation}
without loss of generality, $a=\frac{D\cdot C_0}{-K_X\cdot C_0}$. By construction, $D+aK_X$ intersects non-negatively with every extremal ray of $\NE(X)$, so $D+aK_X$ is nef. By Lemma \ref{l:reduce-to-a=0}, to prove Theorem \ref{thm:conj-plus-K-degree} for $D$, it then suffices to prove the theorem for $D+aK_X$.

The advantage to working with $D+aK_X$ as opposed to $D$ is that $C_0\cdot(D+aK_X)=0$. Let $\pi\colon X\to Y$ be the extremal contraction corresponding to the ray $\RR_{\geq0}C_0$. If $\pi$ is either a Mori fiber space or a divisorial contraction, then there is a nef divisor $D'$ on $Y$ for which $D+aK_X=\pi^*D'$. If $\pi$ is a flipping contraction, then let $\psi\colon X\dasharrow X'$ denote the associated elementary flip. By \cite[Lemma 15.5.7]{CLS}, we have a commutative diagram
\begin{equation}
\label{eqn:flip-diagram}
\begin{gathered}
\xymatrix{
& X^*\ar[rd]^-{\Phi'}\ar[ld]_-{\Phi} & \\
X \ar@{-->}[rr]^{\psi}\ar[dr]_{\pi} & & X'\ar[dl]^-{\pi'}\\
& Y & 
}
\end{gathered}
\end{equation}
such that $X^*$ is a common star subdivision of $X$, $X'$, and $Y$, the maps $\Phi$ and $\Phi'$ are isomorphisms away from the exceptional locus $\exc(\psi)$, and if $D^*$ denotes the torus-invariant divisor on $X^*$ corresponding to the newly inserted ray, then 
\begin{equation}
\label{flip-comparison}
\Phi^*F={\Phi'}^*F'-(F\cdot C_0)D^*
\end{equation}
for all divisors $F$ on $X$ where $F'=\psi_*F$. Letting $D':=\psi_*(D+aK_X)$, equation (\ref{flip-comparison}) tells us $\Phi^*(D+aK_X)={\Phi'}^*D'$. As $\Phi$ and $\Phi'$ are proper and surjective, the fact that $D+aK_X$ is nef implies $\Phi^*(D+aK_X)$ is nef, which in turn implies $D'$ is nef.

To unify notation among these three cases, we denote by $\psi\colon X\dasharrow X'$ the elementary MMP step corresponding to the ray $\RR_{\geq0}C_0$, i.e.~if $\pi$ is a Mori fiber space or a divisorial contraction, we let $X':=Y$ and $\psi:=\pi$; if on the other hand, $\pi$ is a flipping contraction, we let $\psi$ be the associated elementary flip. We have therefore shown that in all three cases, there is a nef divisor $D'$ on $X'$ for which $D+aK_X=\psi^*D'$. If $P$ is not in the exceptional locus, then we would like to apply an inductive strategy to deduce the theorem for $(X,P,D+aK_X)$ from that of $(X',\psi(P),D')$. Proposition \ref{prop:persistence-of-canonical-boundedness} below will allow us to do so.

\begin{lemma}
\label{l:alpha-unchanged-under-pullbacks}
Let $\pi\colon X\to Y$ be a surjective birational morphism of projective $\QQ$-factorial varieties over a number field $k$. Let $P\in X(k)$ be a point which is not in the exceptional locus $\exc(\pi)$ and let $D'$ be a $\QQ$-Cartier divisor on $Y$.  Suppose either that $\{x_i\}$ is a Zariski dense sequence on $X$ converging to $P$ and let $x'_i:=\pi(x_i)$, or suppose $\{x'_i\}$ is a Zariski dense sequence on $X'$ converging to $\pi(P)$ and let $x_i:=\pi^{-1}(x'_i)$ whenever $x'_i\notin\pi(\exc(\pi))$. Then $\alpha_{P,\{x_i\}}(\pi^*D')=\alpha_{\pi(P),\{x'_i\}}(D')$.
\end{lemma}
\begin{proof}
If $\{x_i\}$ is a Zariski dense sequence on $X$ converging to $P$, then only finitely many of the $x_i\in\exc(\pi)$; similarly if $\{x'_i\}$ is a Zariski dense sequence on $X'$ converging to $\pi(P)$, then only finitely many of the $x'_i\in\pi(\exc(\pi))$.  Since the value of $\alpha$ for a sequence is unchanged by removing finitely many elements from the sequence, we may assume $x_i\notin\exc(\pi)$ and $x'_i\notin\pi(\exc(\pi))$ for all $i$. Then $H_{\pi^*D}(x_i)=H_{D'}(x'_i)$. Moreover, the proof of \cite[Proposition~2.4]{MR2} applied to $X\setminus\exc(\pi)$ shows that the distance functions $\dist(P,\cdot)$ and $\dist(\pi(P),\pi(\cdot))$ differ only by a multiplicative factor bounded independently of $P$; note that the cited proposition is stated for only projective varieties, but the proof reduces immediately to compact neighbourhoods of a point. Therefore, it follows directly from the definition of $\alpha$ that $\alpha_{P,\{x_i\}}(\pi^*D')=\alpha_{\pi(P),\{x'_i\}}(D')$.
\end{proof}

\begin{proposition}
\label{prop:persistence-of-canonical-boundedness}
Let $X$ be a projective terminal $\QQ$-factorial split toric variety over a number field $k$ and let $\psi\colon X\dasharrow X'$ be a birational elementary MMP step. If $P\in X(k)\setminus\exc(\psi)$ is a canonically bounded point of $X$, then $\psi(P)$ is a canonically bounded point of $X'$.
\end{proposition}
\begin{proof}
We first consider the case where $\psi$ is a divisorial contraction. Then $\psi$ is a morphism given by the blow-up of $X'$ along a torus-invariant locus $Z$. Let $E\subset X$ be the exceptional divisor. Since $X$ has terminal singularities, $X'$ does as well and so
\[
\psi^*(-K_{X'})=-K_X+aE
\]
for some $a>0$. Given any Zariski dense sequence $\{x'_i\}$ converging to $\psi(P)$, letting $\{x_i\}$ be as in Lemma \ref{l:alpha-unchanged-under-pullbacks}, we find
\begin{align*}
\alpha_{\psi(P),\{x'_i\}}(-K_{X'}) &=\alpha_{P,\{x_i\}}(\psi^*(-K_{X'})) = \alpha_{P,\{x_i\}}(-K_X+aE).
\end{align*}
By concavity of $\alpha$, shown in \cite[Proposition~2.14.(b)]{MR2}, we see
\begin{align*}
\alpha_{\psi(P),\{x'_i\}}(-K_{X'}) =\alpha_{P,\{x_i\}}(-K_X+aE) &\geq \alpha_{P,\{x_i\}}(-K_X)+a\alpha_{P,\{x_i\}}(E)  \\
&>\alpha_{P,\{x_i\}}(-K_X)\geq\dim X=\dim X'
\end{align*}
where the last line follows from the previous one by the effectiveness of $E$ and the fact that $P\not\in E$. 





We next handle the case where $\psi\colon X\dasharrow X'$ is an elementary flip. Let $C_0$ be the generator of the $K_X$-negative ray corresponding to $\psi$. Let $\Phi$, $\Phi'$, and $X^*$ be as in diagram (\ref{eqn:flip-diagram}). Then applying equation (\ref{flip-comparison}) with $F=-K_X$, we have
\[
{\Phi'}^*(-K_{X'}) = \Phi^*(-K_X) + (-K_X.C_0) D^*.
\]
Since $-K_X.C_0>0$, Lemma \ref{l:alpha-unchanged-under-pullbacks} tells us that for any Zariski dense sequence $\{x_i\}$ on $X'$ converging to $\psi(P)$, we have
\begin{align*}
\alpha_{P,\{x_i\}}(-K_{X'}) &=\alpha_{P,\{x_i\}}({\Phi'}^*(-K_{X'}))>\alpha_{P,\{x_i\}}(\Phi^*(-K_X)) \\
&=\alpha_{P,\{x_i\}}(-K_X)\geq\dim X=\dim X'
\end{align*}
where for ease of notation, $P$ and $x_i$ are used to denote points on any of $X$, $X'$, or $X^*$. It follows that $\psi(P)$ is a canonically bounded point of $X'$.
\end{proof}

In light of Proposition \ref{prop:persistence-of-canonical-boundedness} and the discussion beforehand, we employ the following method to prove Theorem \ref{thm:conj-plus-K-degree}. Let $X_1$ be a projective terminal $\QQ$-factorial split toric variety, $D_1$ a nef divisor on $X_1$, and $P_1\in X(k)$ a canonically bounded point. Let $a_1$ be as in equation (\ref{eqn:a-move-to-boundary}). Then $D_1+a_1K_{X_1}\in\mathcal{R}_1$ for some $K_{X_1}$-negative extremal ray $\mathcal{R}_1$ of $\NE(X_1)$. Let $\psi_1\colon X_1\dasharrow X_2$ be the associated elementary MMP step, and let $D_2$ be the nef divisor on $X_2$ such that $D_1+a_1K_{X_1}=\psi_1^*D_2$. Proceeding in this manner, we arrive at the following data: we have a sequence
\[
X_1\stackrel{\psi_1}{\dasharrow} X_2\stackrel{\psi_2}{\dasharrow} \dots \stackrel{\psi_m}{\dasharrow} X_{m+1}
\]
of elementary MMP steps and a sequence of points $P_i\in X_i(k)$ such that $P_i\notin\exc(\psi_i)$ and $P_{i+1}=\psi_i(P_i)$ for $1\leq i<m$, and $P_m\in\exc(\psi_m)$. Furthermore, for $1\leq i\leq m$, we have a nef divisor $D_i$ on $X_i$ and a real number $a_i\geq0$ such that $D_i+a_iK_{X_i}=\psi_i^*D_{i+1}$ is a nef divisor perpendicular to the $K_{X_i}$-negative extremal ray corresponding to $\psi_i$.

Applying Proposition \ref{prop:persistence-of-canonical-boundedness} repeatedly, we see that $P_i$ is a canonically bounded point of $X_i$ for $1\leq i\leq m$. By Lemma \ref{l:reduce-to-a=0}, Theorem \ref{thm:conj-plus-K-degree} for the triple $(X_i,P_i,D_i)$ follows from that of $(X_i,P_i,D_i+a_iK_{X_i})$. So, to prove Theorem \ref{thm:conj-plus-K-degree} for the triple $(X_1,P_1,D_1)$, it suffices to show the result for $(X_m,P_m,D_m)$ and additionally show that the case of $(X_i,P_i,D_i+a_iK_{X_i})$ follows from that of $(X_{i+1},P_{i+1},D_{i+1})$. In other words, we have reduced to proving the following two statements.

\begin{proposition}
\label{prop:main-thm-induction-step}
Let $X$ be a projective terminal $\QQ$-factorial split toric variety over a number field $k$ and let $\psi\colon X\dasharrow X'$ be a birational elementary MMP step corresponding to the extremal ray $\mathcal{R}$. Let $D\in\Nef(X)\cap\mathcal{R}^\perp$ and $D'\in\Nef(X')$ such that $D=\psi^*D'$. If $P\in X(k)\setminus\exc(\psi)$ is canonically bounded and Theorem \ref{thm:conj-plus-K-degree} holds for $(X',\psi(P),D')$ then it holds for $(X,P,D)$.
\end{proposition}

\begin{proposition}
\label{prop:main-thm-base-case}
Let $X$ be a projective terminal $\QQ$-factorial split toric variety over a number field $k$ and let $\psi\colon X\dasharrow X'$ be an elementary MMP step corresponding to the extremal ray $\mathcal{R}$. If $D\in\Nef(X)\cap\mathcal{R}^\perp$ and $P\in X(k)\cap\exc(\psi)$ is canonically bounded, then Theorem \ref{thm:conj-plus-K-degree} holds for $(X,P,D)$.
\end{proposition}

\section{Induction step: $P$ is not in the exceptional locus}
\label{sec:pf-of-main-thms}

In this section, we prove Proposition \ref{prop:main-thm-induction-step}. We assume throughout that $X$ is a projective terminal $\QQ$-factorial split toric variety over a number field $k$, $P\in X(k)\setminus\exc(\psi)$ is a canonically bounded point, and $\psi\colon X\dasharrow X'$ is a birational elementary MMP step corresponding to the contraction of the extremal ray $\mathcal{R}$. We let $C_0$ be the generator of $\mathcal{R}$, $D\in\Nef(X)\cap\mathcal{R}^\perp$, and $D'\in\Nef(X')$ such that $D=\psi^*D'$. We handle the case where $\psi$ is a divisorial contraction in Section \ref{subsec:induction-step-div-contr} and the case where $\psi$ is a flip in Section \ref{subsec:induction-step-flip}.

\subsection{The case of divisorial contractions}
\label{subsec:induction-step-div-contr}

Throughout this subsection, we assume $\psi\colon X\to X'$ is a divisorial contraction. Let $E\subset X$ be the exceptional divisor and $Z\subset X'$ be the torus-invariant locus along which $\psi$ is the blow-up. We first handle the case where $X'\simeq\PP^n$.

\begin{lemma}
\label{l:div-contraction-case-pn}
If $X'\simeq\PP^n$, then there is a smooth irreducible curve $C$ through $P$ such that $-K_X\cdot C\leq\dim X$ and $\alpha_{P,C}(D)\leq \alpha_{P,\{x_i\}}(D)$ for all Zariski dense sequences $\{x_i\}$ on $X$.
\end{lemma}
\begin{proof}
We may assume $X'=\PP^n$. Let $\ell$ be a line in $\PP^n$ that contains both $P$ and at least one point of $Z$. Letting $C$ be the strict transform of $\ell$, we have $C.E\geq 1$. 
Since $K_X=\psi^*K_{\PP^n}+rE$ with $r=\codim(Z)-1$, we have
\[
-K_X\cdot C=-K_{\PP^n}\cdot\psi_*C-rE\cdot C=-K_{\PP^n}\cdot\ell-rE\cdot C\leq n+1-r\leq n.
\]

Next, let $\{x_i\}$ be a Zariski dense sequence on $X$ converging to $P$. By Lemma \ref{l:alpha-unchanged-under-pullbacks}, $\alpha_{P,\{x_i\}}(D)=\alpha_{\psi(P),\{\psi(x_i)\}}(D')$. If $d=\deg(D')$, then Lemma 2.13 and Proposition 2.14 (a) of \cite{MR2} show
\[
\alpha_{P,\{x_i\}}(D)=\alpha_{\psi(P),\{\psi(x_i)\}}(D')=d\alpha_{\psi(P),\{\psi(x_i)\}}(\O(1))\geq d.
\]
On the other hand, since $C$ is smooth at $P$, we have
\[
\alpha_{P,C}(D)=C\cdot D=\ell\cdot D'=d,
\]
proving $\alpha_{P,C}(D)\leq \alpha_{P,\{x_i\}}(D)$.
\end{proof}

Having dispensed with the case where $X'$ is isomorphic to $\PP^n$, we can assume that there is a rational irreducible curve $C'\subseteq X'$ through $\psi(P)$ which is unibranch at $\psi(P)$ such that $-K_{X'}\cdot C'\leq\dim X'=\dim X$ and $$\alpha_{\psi(P),C'}(D')\leq \alpha_{P,\{x'_i\}}(D')$$ for all Zariski dense sequences $\{x'_i\}$ on $X'$. To prove Proposition \ref{prop:main-thm-induction-step} in the case of divisorial contractions, it remains to show the following.

\begin{lemma}
\label{l:div-contraction-case}
Let $C\subseteq X$ be the strict transform of $C'$. Then $C$ is unibranch at $P$, $-K_X\cdot C\leq\dim X$, and $\alpha_{P,C}(D)\leq \alpha_{P,\{x_i\}}(D)$ for all Zariski dense sequences $\{x_i\}$ on $X$.
\end{lemma}
\begin{proof}
Since $X$ has terminal singularities, $K_X=\psi^*K_{X'}+rE$ with $r>0$. Then
\[
-K_X\cdot C=(-\psi^*K_{X'}-rE)\cdot C=-K_{X'}\cdot C'-rE\cdot C\leq\dim X-rE\cdot C\leq\dim X,
\]
where the last inequality follows because $E$ is effective, $C$ is irreducible, and $C$ is not contained in $E$.

Next, let $m$ be the multiplicity of $C$ at $P$. Since $P$ is not in the exceptional locus, $m$ is also the multiplicity of $C'$ at $\psi(P)$. Applying Theorem \ref{thm:unibranch-curve} and using that $C$ and $C'$ are unibranch at $P$ and $\psi(P)$ respectively, we find
\[
\alpha_{P,C}(D)=\frac{1}{m}C\cdot D=\frac{1}{m}C\cdot\psi^*D'=\frac{1}{m}C'\cdot D'=\alpha_{\psi(P),C'}(D').
\]
Now if $\{x_i\}$ is a Zariski dense sequence on $X$ converging to $P$, then Lemma \ref{l:alpha-unchanged-under-pullbacks} shows $\alpha_{P,\{x_i\}}(D)=\alpha_{\psi(P),\{\psi(x_i)\}}(D')$. By the defining property of $C'$, we see $\alpha_{\psi(P),C'}(D')\leq\alpha_{\psi(P),\{\psi(x_i)\}}(D')$, which proves $\alpha_{P,C}(D)\leq\alpha_{P,\{x_i\}}(D)$.
\end{proof}

\subsection{The case of flips}
\label{subsec:induction-step-flip}
In this subsection, we handle the case where $\psi\colon X\dasharrow X'$ is an elementary flip. Since $\psi$ is an isomorphism in codimension $1$, the Picard numbers of $X$ and $X'$ are equal. Since the Picard number of $X$ must be at least $2$, we see then that $X'\not\simeq\PP^n$. So we may assume there is a rational irreducible curve $C'\subseteq X'$ through $\psi(P)$ which is unibranch at $\psi(P)$ such that $-K_{X'}\cdot C'\leq\dim X'=\dim X$ and $\alpha_{\psi(P),C'}(D')\leq \alpha_{P,\{x'_i\}}(D')$ for all Zariski dense sequences $\{x'_i\}$ on $X'$. Let $X^*$, $\Phi$, and $\Phi'$ be as in diagram (\ref{eqn:flip-diagram}). It then suffices to prove the following.

\begin{lemma}
\label{l:flip-case}
Let $\widetilde{C'}\subset X^*$ be the strict transform of $C'$ and $C=\Phi(\widetilde{C'})$. Then $C$ is rational, irreducible, and unibranch at $P$, $-K_X\cdot C\leq\dim X$, and $\alpha_{P,C}(D)\leq \alpha_{P,\{x_i\}}(D)$ for all Zariski dense sequences $\{x_i\}$ on $X$.
\end{lemma}
\begin{proof}
Since $\Phi$ and $\Phi'$ are isomorphisms away from $\exc(\psi)$, and $C'$ is rational and irreducible, it follows that $C$ is as well. Moreover, since $C'$ is unibranch at $P':=\psi(P)$ and $P\notin\exc(\psi)$, we see $C$ is unibranch at $P$.

Next, we see ${\Phi'}^*D'\cdot\widetilde{C'}=D'\cdot\Phi'_*\widetilde{C'}=D'\cdot C'$ and similarly ${\Phi}^*D\cdot\widetilde{C'}=D\cdot C$. Since $D\cdot C_0=0$, equation (\ref{flip-comparison}) tells us ${\Phi}^*D={\Phi'}^*D'$. Let $m$ be the multiplicity of $C$ at $P$. Since $m$ is also the multiplicity of $C'$ at $P'$, we see from Theorem \ref{thm:unibranch-curve} that
\[
\alpha_{P,C}(D)=\frac{1}{m}D\cdot C=\frac{1}{m}D'\cdot C'=\alpha_{P',C'}(D').
\]
Again applying (\ref{flip-comparison}), we find 
\[
K_X\cdot C=\Phi^*K_X\cdot\widetilde{C'}=({\Phi'}^*K_{X'}-(K_X\cdot C_0)D^*)\cdot\widetilde{C'}=K_{X'}\cdot C'-(K_X\cdot C_0)(D^*\cdot\widetilde{C'}).
\]
Recall that $C_0$ generates a $K_X$-negative ray. Since $\widetilde{C'}$ is irreducible and not contained in the effective divisor $D^*$, we have $D^*\cdot\widetilde{C'}\geq0$. By hypothesis, $-K_{X'}\cdot C'\leq\dim X$, so we find $-K_X\cdot C\leq\dim X$.

It remains to show that $\alpha_{P,C}(D)\leq \alpha_{P,\{x_i\}}(D)$ for all Zariski dense sequences $\{x_i\}$ on $X$ converging to $P$. Since $P\notin\exc(\psi)$, only finitely many of the $x_i\in\exc(\psi)$. So, removing these finitely many terms, we may assume $x_i\notin\exc(\psi)$ for all $i$. Let $P^*:=\Phi^{-1}(P)$, $x_i^*:=\Phi^{-1}(x_i)$, and $x'_i:=\Phi'(x_i^*)$. Then two applications of Lemma \ref{l:alpha-unchanged-under-pullbacks} show
\[
\alpha_{P,\{x_i\}}(D)=\alpha_{P^*,\{x_i^*\}}(\Phi^*D)=\alpha_{P^*,\{x_i^*\}}({\Phi'}^*D')=\alpha_{P',\{x'_i\}}(D').
\]
It follows that $\alpha_{P,C}(D)=\alpha_{P',C'}(D')\leq\alpha_{P',\{x'_i\}}(D')=\alpha_{P,\{x_i\}}(D)$.
\end{proof}

\section{Results on fake weighted projective spaces}
The analysis in this section is by far the most involved. Our goal is to prove the following result which forms a crucial step in the proof of Proposition \ref{prop:main-thm-base-case}.

\begin{proposition}
\label{prop:1ps-on-fake-wted-proj-sp}
Let $W$ be a fake weighted projective space with torus $T$, and let $P\in W(k)$. Then there is a unibranch rational curve $C\subseteq W$ through $P$ satisfying the following properties:
\begin{enumerate}
\item\label{prop:1ps-on-fake-wted-proj-sp::curve} There is a $T$-orbit closure $Z\subseteq W$ and a $1$-parameter subgroup $C_0\subseteq T_Z$ of the torus of $Z$ such that $C$ is the closure of $C_0$,
\item\label{prop:1ps-on-fake-wted-proj-sp::n1} $-K_W\cdot C\leq1+\dim W$,
\item\label{prop:1ps-on-fake-wted-proj-sp::n} If $W$ has terminal singularities and is not isomorphic to projective space, then $C$ can be chosen to additionally satisfy $-K_W\cdot C\leq\dim W$.
\end{enumerate}
\end{proposition}

Recall that every fake weighted projective space $W$ admits a canonical toric cover $f\colon W'\to W$ which is \'etale in codimension 1 and such that $W'$ is a weighted projective space, see e.g.~\cite{Bu}. Moreover, there is a subgroup scheme $G=\prod_{i=1}^\ell\mu_{r_i}$ of the torus $T'$ of $W'$ such that under the induced action of $G$, we have $W=W'/G$ and $f$ is the quotient map. The morphism $f$ is referred to as the \emph{universal covering in codimension 1}, and is constructed explicitly as follows. Let $v_0,\dots,v_n\in N$ be the primitive generators for the rays of the fan of $W$. There exist relatively prime positive integers $a_0,\dots,a_n$ such that $\sum a_iv_i=0$ in $N$. The map $f$ corresponds to the finite index inclusion $\iota\colon N'\hookrightarrow N$ where $N'$ is the lattice generated by the $v_i$.

We begin by reducing Proposition \ref{prop:1ps-on-fake-wted-proj-sp} to a subclass of fake weighted projective spaces.

\begin{lemma}
\label{l:reduce-wps-Pn-mod-mup}
If Proposition \ref{prop:1ps-on-fake-wted-proj-sp} (\ref{prop:1ps-on-fake-wted-proj-sp::curve}) and (\ref{prop:1ps-on-fake-wted-proj-sp::n1}) hold for all weighted projective spaces, then they hold for all fake weighted projective spaces.

Furthermore, suppose Proposition \ref{prop:1ps-on-fake-wted-proj-sp} holds for
\begin{enumerate}
\item\label{l:reduce-wps-Pn-mod-mup::wps} weighted projective spaces, and
\item\label{l:reduce-wps-Pn-mod-mup::pnmup} fake weighted projective spaces of the form $\PP^n/\mu_p$, where $p$ is prime and the quotient map $\PP^n\to\PP^n/\mu_p$ is the universal covering in codimension 1.
\end{enumerate}
Then Proposition \ref{prop:1ps-on-fake-wted-proj-sp} holds for all fake weighted projective spaces.
\end{lemma}
\begin{proof}
Let $W$ be a fake weighted projective space. We define a finite surjective toric morphism $g\colon W'\to W$ which is \'etale in codimension $1$ as follows. If the universal covering in codimension $1$ of $W$ is not isomorphic to projective space, then we take $g\colon W'\to W$ to be the universal covering in codimension $1$. If, on the other hand, $f\colon\PP^n\to W$ is the universal covering in codimension $1$ realizing $W$ as $\PP^n/G$, then choose a prime $p$ and a subgroup scheme $\mu_p\subseteq G$. The map $f$ then factors as $\PP^n\to W':=\PP^n/\mu_p\stackrel{g}{\to} W$. Since $f$ is finite surjective and \'etale in codimension $1$, the map $g$ is as well.

Since in either case $g\colon W'\to W$ is toric finite surjective, the induced map on lattices $N'\to N$ is a finite index inclusion which induces a bijection between the cones in the fans $\Sigma_W$ and $\Sigma_{W'}$. Let $T$ and $T'$ denote the tori of $W$ and $W'$, respectively. Given $P\in W(k)$, choose a lift $P'\in W'(k)$. By \cite[Proposition 9.3.3]{sing-fano-pic1}, if $W$ has terminal singularities, then $W'$ does as well. 

By hypothesis, there is a $T'$-orbit closure $Z'\subseteq W'$ and a $1$-parameter subgroup $C'_0\subseteq T_{Z'}$ such that its closure $C'\subseteq W'$ contains $P'$ and satisfies $-K_{W'}\cdot C'\leq 1+\dim W'=1+\dim W$ or $-K_{W'}\cdot C'\leq \dim W$, depending on whether $W'$ has terminal singularities. Now, $Z'$ corresponds to a cone $\sigma\in\Sigma_{W'}$. Since $\sigma$ can also be considered as a cone of $\Sigma_W$ on the finer lattice $N$, we obtain a $T$-orbit closure $Z\subseteq W$ and a toric map $g|_{Z'}\colon Z'\to Z$. Since $C'$ is the closure of a $1$-parameter subgroup of $T_{Z'}$, its image $C:=f(C')$ is the closure of a $1$-parameter subgroup of $T_Z$. In particular, $C$ is unibranch and contains $P$. Since $g$ is \'etale in codimension 1, we have $g^*K_W=K_{W'}$. Letting $d$ denote the degree of $g|_{C'}\colon C'\to C$, we find
\[
-K_W\cdot C=\frac{1}{d} (-K_W)\cdot g_*C'=\frac{1}{d} (-K_{W'})\cdot C'\leq -K_{W'}\cdot C',
\]
thereby yielding the desired bound for $-K_W\cdot C$.
\end{proof}

The following lemma provides a bound that is useful throughout the rest of this section.

\begin{lemma}
\label{l:all-D-C-leq1}
If $W$ is a weighted projective space and $P\in W(k)$. Then there is a curve $C\subseteq W$ through $P$ satisfying property (\ref{prop:1ps-on-fake-wted-proj-sp::curve}) of Proposition \ref{prop:1ps-on-fake-wted-proj-sp}, and such that $D\cdot C\leq1$ for all torus-invariant divisors $D$ on $W$.
\end{lemma}
\begin{proof}
Let $v_0,\dots,v_n\in N$ be the primitive generators for the rays of $\Sigma_W$, and let $a_0,\dots,a_n$ be relatively prime positive integers with $\sum a_iv_i=0$ in $N$. Without loss of generality $a_0=\max(a_i)$. Since $W$ is a weighted projective space, $N$ is the lattice spanned by the $v_i$. Let $D_i$ be the torus-invariant divisor corresponding to $v_i$. We prove the result by inducting on dimension.

We first handle the base case where $\dim W=1$, i.e.~$W=\PP^1$. Then choosing $C=W$, we find $C\cdot D_i=\deg(D_i)=1$.

Next, we handle the case where $P\in T$ or where $P$ is in the torus $T_{D_0}$ of $D_0$. Let $C$ be the closure of the $1$-parameter subgroup corresponding to the lattice point $v_0\in N$. Let $\phi$ denote the unique function $\phi\colon N_\RR\to\RR$ which is linear on all maximal cones subject to the condition $\phi(v_0) = 1$ and $\phi(v_i)=0$ for $i\neq0$. Then
\[
D_0\cdot C=\phi(v_0)+\phi(-v_0).
\]
Since $-v_0=\sum_{i>0}\frac{a_i}{a_0}v_i$ is in the maximal cone generated by $v_1,\dots,v_n$, we see $\phi(-v_0)=0$ and so $D_0\cdot C=1$. Furthermore, since $\frac{1}{a_0}D_0$ and $\frac{1}{a_i}D_i$ are linearly equivalent for all $i$, and $a_i\leq a_0$, we find
\[
D_i\cdot C=\frac{a_i}{a_0}D_0\cdot C\leq1.
\]
Note that $C$ contains both the identity of $T$ and the identity of $T_{D_0}$. Thus, if $P\in T$ or $P\in T_{D_0}$, a suitable $T$-translate of $C$ contains $P$.

It remains to handle the case where $P\in D_j$ for some $j\neq0$. Now, $D_j$ is a weighted projective space of dimension $\dim W-1$; its lattice is given by $\overline{N}:=N/\ZZ v_j$ and its torus-invariant divisors $D'_i$ correspond to the ray spanned by $v_i$ in $\overline{N}$ for $i\neq j$. By induction, there exists a curve $C\subseteq D_j$ which is the closure of a $1$-parameter subgroup in a $T_{D_j}$-orbit closure; in particular, $C$ is also the closure of a $1$-parameter subgroup in a $T$-orbit closure. By construction $D'_i\cdot C\leq1$ for all $i\neq j$. Letting $m_{ij}\geq1$ denote the multiplicity of the cone $\<v_i,v_j\>$ in $N$, we have from \cite[p.~100]{fulton-tv} that
\[
D_i\cdot C=\frac{1}{m_{ij}}D'_i\cdot C\leq1
\]
for $i\neq j$. To handle the case of $D_j$, we apply the same technique as above:
\[
D_j\cdot C=\frac{a_j}{a_0}D_0\cdot C\leq D'_0\cdot C\leq1.
\]
This completes the proof of the result.
\end{proof}

Applying Lemmas \ref{l:reduce-wps-Pn-mod-mup} and \ref{l:all-D-C-leq1}, we are able to handle many cases of Proposition \ref{prop:1ps-on-fake-wted-proj-sp}.

\begin{corollary}
\label{cor:1ps-on-fake-wted-proj-sp-not-covered-by-all1s}
The following are true:
\begin{enumerate}[(1)]
\item\label{cor:1ps-on-fake-wted-proj-sp-not-covered-by-all1s::wps-abc} Proposition \ref{prop:1ps-on-fake-wted-proj-sp} holds for weighted projective spaces.

\item\label{cor:1ps-on-fake-wted-proj-sp-not-covered-by-all1s::fwps-ab} Proposition \ref{prop:1ps-on-fake-wted-proj-sp} (\ref{prop:1ps-on-fake-wted-proj-sp::curve}) and (\ref{prop:1ps-on-fake-wted-proj-sp::n1}) hold for all fake weighted projective spaces.
\end{enumerate}
\end{corollary}
\begin{proof}
By Lemma \ref{l:reduce-wps-Pn-mod-mup}, statement (\ref{cor:1ps-on-fake-wted-proj-sp-not-covered-by-all1s::fwps-ab}) follows from statement (\ref{cor:1ps-on-fake-wted-proj-sp-not-covered-by-all1s::wps-abc}). 

Let $W$ be a weighted projective space. We let $n=\dim W$ and again denote by $v_0,\dots,v_n\in N$ the primitive generators for the rays of $\Sigma_W$. Let $a_0,\dots,a_n$ be relatively prime positive integers with $\sum a_iv_i=0$ in $N$. Without loss of generality $a_0=\max(a_i)$. 

By Lemma \ref{l:all-D-C-leq1}, there is a curve $C\subseteq W$ through $P$ satisfying property (\ref{prop:1ps-on-fake-wted-proj-sp::curve}) of Proposition \ref{prop:1ps-on-fake-wted-proj-sp}, and such that $D_i\cdot C\leq1$ for all $i$. Since $\frac{1}{a_0}D_0$ and $\frac{1}{a_i}D_i$ are linearly equivalent for all $i$, and $D_0\cdot C\leq1$, we see
\[
-K_{W}\cdot C=\frac{1}{a_0}(\sum_{i=0}^n a_i)D_0\cdot C\leq\frac{1}{a_0}\sum_{i=0}^n a_i.
\]
As $a_0=\max(a_i)$, we see $\frac{1}{a_0}\sum_{i=0}^n a_i\leq n+1$, thereby proving Proposition \ref{prop:1ps-on-fake-wted-proj-sp} (\ref{prop:1ps-on-fake-wted-proj-sp::n1}) for weighted projective spaces.

It remains to prove that if $W$ has terminal singularities and is not isomorphic to $\PP^n$, then $\frac{1}{a_0}\sum_{i=0}^n a_i\leq n$. For ease, of notation, let $h=\sum_{i=0}^n a_i$. By \cite[Proposition 2.3]{classifying-terminal-wps}, we see
\begin{equation}
\label{eqn:terminal-wps-wts}
\sum_{i=0}^n\left\{\frac{a_i\kappa}{h}\right\}\leq n-1
\end{equation}
for all $2\leq\kappa\leq h-2$, where $\{x\}=x-\lfloor x\rfloor$. Since $W\not\simeq\PP^n$, we know each $a_i\geq1$ and $a_0\geq2$; in particular, we can choose $\kappa=n$.

Now, if $\frac{1}{a_0}\sum_{i=0}^n a_i>n$, then $\frac{na_0}{h}<1$, and so $\lfloor\frac{na_i}{h}\rfloor=0$ for all $i$. As a result, 
\[
\sum_{i=0}^n\left\{\frac{na_i}{h}\right\}=\sum_{i=0}^n \frac{na_i}{h}=n,
\]
contradicting (\ref{eqn:terminal-wps-wts}).
\end{proof}

In light of Lemma \ref{l:reduce-wps-Pn-mod-mup} and Corollary \ref{cor:1ps-on-fake-wted-proj-sp-not-covered-by-all1s}, to finish the proof of Proposition \ref{prop:1ps-on-fake-wted-proj-sp}, it remains to handle the case where $W$ has terminal singularities and is of the form given in Lemma \ref{l:reduce-wps-Pn-mod-mup} (\ref{l:reduce-wps-Pn-mod-mup::pnmup}). We first handle the case where $P \in T$ through Lemma \ref{l:1ps-on-fake-wted-proj-sp-all1-wts-P-in-T} and Corollary \ref{cor:1ps-on-fake-wted-proj-sp-all1-wts-P-in-T}.

\begin{lemma}
\label{l:1ps-on-fake-wted-proj-sp-all1-wts-P-in-T}
Let $W=\PP^n/\mu_r$ be a fake weighted projective space where the quotient map $\PP^n\to W$ is the universal covering in codimension 1. Then there is a standard affine patch $x_j\neq 0$ of $\PP^n$ on which the action of $\zeta\in\mu_r$ is given by \[\left[\zeta^{w_0}x_0:\ldots:\zeta^{w_{j-1}}x_{j-1}:x_j:\zeta^{w_{j+1}}x_{j+1}:\ldots:\zeta^{w_n}x_n\right]\] 
such that $w_i\leq\frac{rn}{n+1}$ for all $i$.
\end{lemma}
\begin{proof}
In what follows, we will denote by $M(k)$ the unique element of $\{0,1,\ldots,r-1\}$ that is congruent to $k$ modulo $r$.

First, note that we may reorder the coordinates of $\PP^n$ so that the action of $\mu_r$ on $\PP^n$ globally is given by
\[[x_0:\zeta^{w_1}x_1:\ldots:\zeta^{w_n}x_n]\]
where the $w_i$ are positive integers satisfying $r:=w_0>w_1\geq\ldots\geq w_n\geq w_{n+1}:=0$.  For $0\leq j\leq n$, if we identify the the $j$-th affine patch $x_j=1$ with $\AA^n$, the action of $\zeta\in\mu_r$ is given by
\[\left(\zeta^{w_0-w_j}x_0,\ldots,\zeta^{w_{j-1}-w_j}x_{j-1},\zeta^{w_{j+1}-w_j}x_{j+1},\ldots,\zeta^{w_n-w_j}x_n\right).\]

Next, notice that $\sum_{j=0}^n M(w_j-w_{j+1})=(w_0-w_1)+\dots+(w_{n-1}-w_n)+(w_n-w_{n+1})=r$. So by the Pigeonhole Principle, there is some $j$ for which 
\[w_j-w_{j+1}=M(w_j-w_{j+1})\geq\frac{r}{n+1}.\] 
In particular, $w_j\geq w_j-w_{j+1}\geq\frac{r}{n+1}$. Furthermore, $w_j>w_{j+1}$ since otherwise $r=0$, a contradiction.

On the $j$-th affine patch, the weights of the $\mu_r$-action are given by $M(w_i-w_j)$ for $i\neq j$. If $i<j$, then $M(w_i-w_j)=w_i-w_j$; since $w_i<r$ and $w_j\geq\frac{r}{n+1}$, we find $M(w_i-w_j)<\frac{rn}{n+1}$. If $i>j$, then since $w_j>w_{j+1}\geq w_i$, we find $M(w_i-w_j)=r+w_i-w_j\leq r+w_{j+1}-w_j\leq\frac{rn}{n+1}$, as desired.
\end{proof}

\begin{corollary}
\label{cor:1ps-on-fake-wted-proj-sp-all1-wts-P-in-T}
Proposition \ref{prop:1ps-on-fake-wted-proj-sp} holds for fake weighted projective spaces $W$ of the form given in Lemma \ref{l:reduce-wps-Pn-mod-mup} (\ref{l:reduce-wps-Pn-mod-mup::pnmup}) whenever $P\in T$. In fact, the stronger conclusion $-K_W\cdot C\leq\dim W$ holds even if $W$ does not have terminal singularities.
\end{corollary}
\begin{proof}
Let $W=\PP^n/\mu_p$ be a fake weighted projective space where the quotient map $f\colon\PP^n\to W$ is the universal covering in codimension $1$. By Lemma \ref{l:1ps-on-fake-wted-proj-sp-all1-wts-P-in-T}, after permuting coordinates, we may assume that on the standard affine patch $x_0\neq1$, $\zeta\in\mu_p$ acts by $(\zeta^{w_1}x_1,\ldots,\zeta^{w_n}x_n)$ with
\[
w_n\leq\dots\leq w_1\leq\frac{np}{n+1}.
\]

Since the restriction of $f\colon\PP^n\to W$ to the torus $T=\GG_m^n\subseteq W$ is a $\mu_r$-torsor, giving a $1$-parameter subgroup $\GG_m\to T$ is equivalent to giving a diagram
\[
\xymatrix{
\GG_m\ar[r]^-{\gamma}\ar[d]_-{\beta} & \GG_m^n\\
\GG_m & 
}
\]
where $\beta$ is a $\mu_p$-torsor and $\gamma$ is a $\mu_p$-equivariant map. In particular, we can take $\beta$ and $\gamma$ to be the maps $\beta(t)=t^p$ and $\gamma(t)=(t^{w_1},\dots,t^{w_n})$. Let $C\subseteq W$ be the closure of the $1$-parameter subgroup defined by the diagram, and let $C'\subseteq\PP^n$ be the closure of the $1$-parameter subgroup defined by $\gamma$. We then have $f(C')=C$, and since $\beta$ is a degree $p$ map, we see $f_*C'=pC$. As $w_1=\max(w_i)$, we have $-K_{\PP^n}\cdot C'=(n+1)w_1$. Since $f$ is \'etale in codimension 1, $f^*K_W=K_{\PP^n}$ and so
\[
-K_W\cdot C'=\frac{1}{p}(-K_W)\cdot f_*C=\frac{1}{p}(-K_{\PP^n})\cdot C'=\frac{n+1}{p}w_1\leq n,
\]
proving our desired inequality.
\end{proof}

We now turn to the case where $P$ lives on the boundary of $W$, which is handled in Lemma \ref{l:characterization-boundary-Pn-mod-mup} and Corollary \ref{cor:1ps-on-fake-wted-proj-sp-all1-wts-P-boundary}.

\begin{lemma}
\label{l:characterization-boundary-Pn-mod-mup}
Let $p$ be a prime and $W=\PP^n/\mu_p$ a fake weighted projective space such that the quotient map $\PP^n\to W$ is the universal covering in codimension 1. If $D\subseteq W$ is a torus-invariant divisor, then either $D\simeq\PP(1,\dots,1,p,\dots,p)$, or $D\simeq\PP^{n-1}/\mu_p$ is a fake weighted projective space such that the quotient map $\PP^{n-1}\to D$ is the universal covering in codimension 1.

Furthermore, if $D$ is a weighted projective space, then there is a torus-invariant divisor $D'\neq D$ such that the cone in the fan $\Sigma_W$ corresponding to $D\cap D'$ has multiplicity strictly greater than $1$.
\end{lemma}
\begin{proof}
Let $v_0,\dots,v_n\in N$ be the primitive generators for the rays of $\Sigma_W$ and let $N'=\ZZ v_0+\dots+\ZZ v_n$. By hypothesis, $[N:N']=p$ and $\sum_{i=0}^n v_i=0$. The fan for $D_n$ lives on the lattice $\overline{N}:=N/\ZZ v_n$; its rays are generated by the images $\overline{v}_i\in\overline{N}$ of the $v_i$ for $0\leq i<n$. Let $b_i\in\ZZ^+$ and $\overline{v}'_i\in\overline{N}$ be the primitive lattice point such that $\overline{v}_i=b_i\overline{v}'_i$. Letting $\overline{N}':=N'/\ZZ v_n$, we see the induced map $N/N'\to\overline{N}/\overline{N}'$ is an isomorphism, and hence
\[
[\overline{N}:\overline{N}']=p.
\]
Let $\overline{N}'_0:=\ZZ\overline{v}'_0+\dots+\ZZ\overline{v}'_{n-1}$, and note that the universal covering in codimension $1$ of $D_n$ is induced by the inclusion of lattices $\overline{N}'_0\subseteq\overline{N}$. So, $D_n$ is a weighted projective space if and only if $\overline{N}'_0=\overline{N}$.

From the inclusions $\overline{N}'\subseteq\overline{N}'_0\subseteq\overline{N}$ and the fact that $[\overline{N}:\overline{N}']=p$, we see $D_n$ is not a weighted projective space if and only if $\overline{N}'_0=\overline{N}'$. Since $v_0=-\sum_{i=1}^n v_i$, we see that $v_1,\dots,v_n$ is a $\ZZ$-basis for $N'$ and so $\overline{v}_1,\dots,\overline{v}_{n-1}$ is a $\ZZ$-basis for $\overline{N}'$. Now, if $\overline{N}'_0=\overline{N}'$, then $\overline{v}'_1=\sum_{i=1}^{n-1}c_i\overline{v}_i$ for some $c_i\in\ZZ$. As a result, $\overline{v}_1=\sum_{i=1}^{n-1}b_1c_i\overline{v}_i$, so $b_1=1$. Similarly, all $b_i=1$, so $\sum_{i=0}^{n-1}\overline{v}'_i=0$, i.e.~$\Sigma_{\overline{N}'_0}$ is the fan for $\PP^{n-1}$ so $\PP^{n-1}\to D_n$ is the universal covering in codimension $1$, identifying $D_n$ with $\PP^{n-1}/\mu_p$.

We may therefore assume that $D_n$ is a weighted projective space. In order to show $D_n\simeq\PP(1,\dots,1,p,\dots,p)$, it is equivalent to show that every maximal cone of $D_n$ has multiplicity dividing $p$. Given such a maximal cone $\sigma$, after reindexing we can assume $\sigma=\<\overline{v}'_1,\dots,\overline{v}'_{n-1}\>$. Since $\overline{v}_0=-\sum_{i=1}^{n-1}\overline{v}_i$, we see $\overline{N}'=\ZZ\overline{v}_1+\dots+\ZZ\overline{v}_{n-1}$. From the inclusions
\[
\overline{N}'\subseteq \ZZ\overline{v}'_1+\dots+\ZZ\overline{v}'_{n-1}\subseteq \overline{N}
\]
and the fact that $[\overline{N}:\overline{N}']=p$, we see
\[
\mult(\sigma)=[\overline{N}:\ZZ\overline{v}'_1+\dots+\ZZ\overline{v}'_{n-1}]\in\{1,p\},
\]
as desired.

Lastly, note that $b_i=\mult(\<v_i,v_n\>)$ for $0\leq i<n$. If all $b_i=1$, then $\overline{N}'=\overline{N}'_0$, which, as we have observed above, is equivalent to the statement that $D_n$ is not a weighted projective space. So, if $D_n$ is a weighted projective space, then there must exist some $i < n$ for which the cone corresponding to $D_i \cap D_n$ has multiplicity $b_i > 1$.
\end{proof}

\begin{corollary}
\label{cor:1ps-on-fake-wted-proj-sp-all1-wts-P-boundary}
Proposition \ref{prop:1ps-on-fake-wted-proj-sp} holds for fake weighted projective spaces $W$ of the form given in Lemma \ref{l:reduce-wps-Pn-mod-mup} (\ref{l:reduce-wps-Pn-mod-mup::pnmup}). In fact, the stronger conclusion $-K_W\cdot C\leq\dim W$ holds even if $W$ does not have terminal singularities.
\end{corollary}
\begin{proof}
We prove the statement by induction on $\dim W$. Let $W=\PP^n/\mu_p$ as in Lemma \ref{l:reduce-wps-Pn-mod-mup} (\ref{l:reduce-wps-Pn-mod-mup::pnmup}), and let $P\in W(k)$. Let $v_0,\dots,v_n\in N$ be the primitive generators for the rays of $\Sigma_W$, and denote by $D_i$ the torus-invariant divisor corresponding to $v_i$. If $P\in T$, then the statement follows from Corollary \ref{cor:1ps-on-fake-wted-proj-sp-all1-wts-P-in-T}. So, we may assume without loss of generality that $P\in D_0$. For $1\leq i\leq n$, let $D'_i$ denote the torus-invariant divisor on $D_0$ corresponding to $v_i$.

First suppose that $D_0$ is a weighted projective space. Then by Lemma \ref{l:characterization-boundary-Pn-mod-mup}, there exists $i\neq0$ such that the multiplicity of the cone $\<v_0,v_i\>$ is $m\geq2$. Since $D_0$ is a weighted projective space, Lemma \ref{l:all-D-C-leq1} yields a curve $C\subseteq D_0$ satisfying Proposition \ref{prop:1ps-on-fake-wted-proj-sp} (\ref{prop:1ps-on-fake-wted-proj-sp::curve}) and $C.D'_j \leq 1$ for all $j$. Then
\[
-K_W\cdot C=(n+1)D_i\cdot C=\frac{n+1}{m}D'_i\cdot C\leq\frac{n+1}{m}\leq n.
\]
Note, in particular, that this handles the base case of our induction. Indeed, there are no $1$-dimensional fake weighted projective spaces of the form given in Lemma \ref{l:reduce-wps-Pn-mod-mup} (\ref{l:reduce-wps-Pn-mod-mup::pnmup}), so the base case is $n=2$, in which case we necessarily have $D_0\simeq\PP^1$.

If $D_0$ is not a weighted projective space, then by Lemma \ref{l:characterization-boundary-Pn-mod-mup}, we know $D_0\simeq\PP^{n-1}/\mu_p$ as in Lemma \ref{l:reduce-wps-Pn-mod-mup} (\ref{l:reduce-wps-Pn-mod-mup::pnmup}). By induction on dimension, we can assume the existence of our desired $C$ with $-K_{D_0}\cdot C\leq n-1$. Since $\sum_{i=1}^n D'_i\cdot C=-K_{D_0}\cdot C$, by the Pigeonhole Principle, we may without loss of generality that $D'_1\cdot C\leq\frac{n-1}{n}$. So, for all $0\leq i\leq n$, we have $D_i\cdot C=D_1\cdot C\leq D'_1\cdot C\leq\frac{n-1}{n}$, which implies $-K_W\cdot C\leq\frac{1}{n}(n-1)(n+1)\leq n$.
\end{proof}

Putting these results to together we have:

\begin{proof}[{Proof of Proposition \ref{prop:1ps-on-fake-wted-proj-sp}}]
Corollary \ref{cor:1ps-on-fake-wted-proj-sp-not-covered-by-all1s} shows that Proposition \ref{prop:1ps-on-fake-wted-proj-sp} (\ref{prop:1ps-on-fake-wted-proj-sp::curve}) and (\ref{prop:1ps-on-fake-wted-proj-sp::n1}) hold for all fake weighted projective spaces and that part (\ref{prop:1ps-on-fake-wted-proj-sp::n}) additionally holds for all weighted projective spaces. By Lemma \ref{l:reduce-wps-Pn-mod-mup}, it remains to show Proposition \ref{prop:1ps-on-fake-wted-proj-sp} holds for fake weighted projective spaces of the form given in Lemma \ref{l:reduce-wps-Pn-mod-mup} (\ref{l:reduce-wps-Pn-mod-mup::pnmup}). This is handled in Corollary \ref{cor:1ps-on-fake-wted-proj-sp-all1-wts-P-boundary}.
\end{proof}

\section{Base case: $P$ is in the exceptional locus}
\label{sec:base-case-PinExc}

In this section we prove Proposition \ref{prop:main-thm-base-case}, thereby finishing the proof of Theorem \ref{thm:conj-plus-K-degree}, and hence also proving Theorem \ref{thm:conj-higher-dim}. We begin with a lemma that allows us to reduce to the case of fake weighted projective spaces.

\begin{lemma}
	\label{l:reduction-to-fibers}
	Let $X$ be a projective terminal $\QQ$-factorial split toric variety and let $\pi\colon X\to Y$ be an elementary contraction corresponding to the extremal ray $\mathcal{R}$. Suppose $C\subseteq X$ is a curve contracted by $\pi$. If $F$ is the reduction of the fiber of $\pi$ containing $C$, then $-K_X\cdot C\leq -K_F\cdot C$.
\end{lemma}
\begin{proof}
	Let $v_1,\dots,v_\ell$ be the rays of the fan $\Sigma_X$, and $D_i\subseteq X$ denote the torus-invariant divisor corresponding to $v_i$. Let $y=\pi(F)$. There is a unique torus-orbit closure $Z\subseteq Y$ such that $y$ is contained in the torus $T_Z$ of $Z$. Since the fibers of $\pi$ are irreducible, $\pi^{-1}(Z)$ is also irreducible. So, the reduction of $\pi^{-1}(Z)$ is a torus-orbit closure $W\subseteq X$. Let $\tau\in\Sigma_X$ be the cone corresponding to $W$. Since $F$ is positive-dimensional and it is a general fiber of $\pi|_W\colon W\to Z$, we see $W$ is contained in the exceptional locus $\exc(\pi)$.
	
	By \cite[Corollary 14-2-2]{matsuki}, $\exc(\pi)$ is the torus-orbit closure corresponding to the cone spanned by the rays $v_i$ with $D_i\cdot C<0$. Furthermore, \cite[Corollary 14-2-2]{matsuki} shows that the toric map from $\exc(\pi)$ to its image corresponds to the quotient map $\eta\colon N/N_-\to N/N_{\neq0}$ where $N_-$ (resp.~$N_{\neq0}$) is the saturation of the sublattice generated by the $v_i$ with $D_i\cdot C<0$ (resp.~$D_i\cdot C\neq0$). So if $D_i\cdot C>0$, then $\eta(\overline{v}_i)=0$ and hence $D_i$ does not contain a fiber of $\pi$. 
	In particular, $D_i\cdot C\leq0$ if $v_i$ is a ray of $\tau$.
	
	Reordering the rays if necessary, we can assume $\tau=\<v_1,\dots,v_a\>$, and that $v_{a+1},\dots,v_b$ are the rays in $\Star(\tau)$ which are not in $\tau$. For $a<i\leq b$, we let $D'_i$ be the torus-invariant divisor on $W$ corresponding to $v_i$, and let $m_i\geq1$ be the multiplicity of the cone $\<v_1,\dots,v_a,v_i\>$. Since $D_i\cdot C=0$ for $v_i\notin\Star(\tau)$, and since $D_i\cdot C\leq0$ for $v_i\in\tau$,
	\[
	-K_X\cdot C=\sum_{i=1}^b D_i\cdot C\leq\sum_{i=a+1}^b D_i\cdot C=\sum_{i=a+1}^b \frac{1}{m_i} D'_i\cdot C\leq\sum_{i=a+1}^b D'_i\cdot C=-K_W\cdot C.
	\]
	Finally, since $F$ is a general fiber of $\pi|_W$, we find $K_W|_F=K_F$, and so $-K_X\cdot C\leq -K_F\cdot C$.
\end{proof}

If $X\simeq\PP^n$, then Theorem \ref{thm:conj-plus-K-degree} follows from \cite[Theorem 2.6]{McK}. Therefore, the following result finishes the proof of Proposition \ref{prop:main-thm-base-case}.

\begin{proposition}
\label{prop:Mori-fiber-sp-case}
Let $X$ be a projective terminal $\QQ$-factorial split toric variety over a number field $k$ and let $\pi\colon X\to Y$ be the elementary contraction corresponding to an extremal ray $\mathcal{R}$. Let $D\in\Nef(X)\cap\mathcal{R}^\perp$ and $P\in X(k)\cap\exc(\pi)$ a canonically bounded point. If $X\not\simeq\PP^n$, then there exists an irreducible rational curve $C$ through $P$ such that $C$ is unibranch at $P$, $-K_X\cdot C\leq\dim X$, and $\alpha_{P,C}(D)\leq \alpha_{P,\{x_i\}}(D)$ for all Zariski dense sequences $\{x_i\}$ on $X$.
\end{proposition}
\begin{proof}
Let $F$ be the reduction of the fiber containing $P$. It follows from 
\cite[Remark 3.3]{equiv-cmpt} that $F$ is a fake weighted projective space. 

Suppose first that $Y$ is a point. Then $F=X$. Since $X$ has terminal singularities and is not isomorphic to projective space, Proposition \ref{prop:1ps-on-fake-wted-proj-sp} (\ref{prop:1ps-on-fake-wted-proj-sp::curve}) and (\ref{prop:1ps-on-fake-wted-proj-sp::n}) tell us there is an irreducible rational curve $C$ through $P$ which is unibranch at $P$ and satisfies $-K_X\cdot C\leq\dim X$.

Next suppose $Y$ is not a point. Then $\dim F\leq\dim X-1$. By Proposition \ref{prop:1ps-on-fake-wted-proj-sp} (\ref{prop:1ps-on-fake-wted-proj-sp::curve}) and (\ref{prop:1ps-on-fake-wted-proj-sp::n1}), there is an irreducible rational curve $C\subseteq F$ through $P$ which is unibranch at $P$ and satisfies $-K_F\cdot C\leq 1+\dim F$. By Lemma \ref{l:reduction-to-fibers}, we have
\[
-K_X\cdot C\leq -K_F\cdot C\leq1+\dim F\leq\dim X.
\]

Lastly, note that, regardless of whether or not $Y$ is a point, there is a nef divisor $D'$ on $Y$ for which $D=\psi^*D'$. Then by Theorem \ref{thm:unibranch-curve}, if $m$ denotes the multiplicity of $C$ at $P$, we have $\alpha_{P,C}(D)=\frac{1}{m}C\cdot D=\frac{1}{m}\psi_*(C)\cdot D'=0$, so $\alpha_{P,C}(D)=0\leq\alpha_{P,\{x_i\}}(D)$ for all Zariski dense sequences $\{x_i\}$ on $X$.
\end{proof}

\section{Finding the curve of best approximation}
\label{sec:subtlety}

A curve $C\subseteq X$ is said to be a \emph{curve of best approximation} with respect to $D$ if $\alpha_{P,C}(D)=\alpha_P(D)$. The curve $C$ constructed in Theorem \ref{thm:conj-higher-dim} is not required to be a curve of best approximation, but only one that approximates $P$ better than any Zariski dense sequence. In addition to the theoretical point raised in Remark \ref{rmk:higher-dimensional-conj} that there may be some Zariski-degenerate sequence with higher dimensional closure that approximates $P$ better than $C$, there is the very practical point mentioned in Remark~\ref{rmk:not-best-curve} that the curve $C$ we find is {\em in fact} not always a curve of best approximation to $P$, as we discuss in this section.

For example, let $k$ be a number field, and fix a place $v$ of $k$. If $X$ is the weighted projective space $\PP(4,7,13)$, and $D$ is the generator of the Picard group of $X$, then $D=-\frac{4\cdot7\cdot13}{4+7+13}K_X$. Assuming canonical boundedness of the point $P=[1:1:1]$, we find then that $\alpha_{P,\{x_i\}}(D)\geq\frac{91}{3}$. Our proof of Theorem \ref{thm:conj-higher-dim} for this choice of $X$ and $P$ ultimately comes from Lemma \ref{l:all-D-C-leq1}. Specifically, the curve $C$ we construct in this case is $x^7=y^4$, which has $D$-degree $28$. So, $\alpha_{P,C}(D)=28<\frac{91}{3}\leq\alpha_{P,\{x_i\}}(D)$ for all Zariski dense sequences $\{x_i\}$. However, it is easy to see (for this particular $X$ and $P$) that there are other curves which have smaller $\alpha$-value, e.g.~the curve $x^5=yz$ has $D$-degree $20$ and hence has $\alpha$-value $20$.

One may wonder whether $x^5=yz$ is the curve of best approximation. The answer turns out to be interesting: a thorough search reveals that the curve $C'$ given by the equation
\[
x^8y+xy^5-3x^3y^2z+z^3=0
\]
satisfies $\alpha_{P,C'}(D)=19.5$ provided that $\sqrt{-3}\in k_v\setminus k$.  This is because $C'$, which has degree $C'\cdot D=39$, is singular at $P$ and the tangent directions to $C'$ at $P$ are distinct and split over the field $\Q(\sqrt{-3})$. So, if $\sqrt{-3}\in k_v\setminus k$, then Theorem~\ref{thm:curve} implies $\alpha_{P,C'}(D)=\frac{39}{2}=19.5$.

In fact, as we now explain, the curve $C'$ is a curve of best approximation to $P$ when $\sqrt{-3}\in k_v\setminus k$. Let $D'$ be the Weil divisor $x=0$, so that $D=91D'$. Letting $\m_P$ be the maximal ideal at $P$, we see $\O_{X,P}/\m_P^3$ has dimension ${4\choose 2}=6$. A straightforward computation shows that $H^0(14D')$ has dimension $7>{4\choose 2}$, and so there must be some non-zero section $g\in H^0(14D')$ that vanishes at $P$ to order at least $3$. In fact, one computes that, up to scalar, there is a unique such $g$, which defines the curve $C''$
\[
x^{14}-4x^9yz+x^7y^4+6x^4y^2z^2-4x^2y^5z+y^8-xz^4=0.
\]
Then the section $g^{13}\in H^0(13\cdot14D')=H^0(2D)$ vanishes at $P$ to order at least $39$. Thus, if $\pi\colon Y\to X$ denotes the blowup of $X$ at $P$, with exceptional divisor $E$, then $2\pi^*D-39E$ is effective. Let $B\subset X$ be the image of the asymptotic base locus of $2\pi^*D-39E$. Then Theorem~3.3 of \cite{MR}, shows that for any sequence $\{x_i\}$ not contained in $B$, we have $\alpha_{P,\{x_i\}}(D)\geq \frac{39}{2}=19.5$.  Thus, unconditionally (i.e.~without even assuming $P$ is canonically bounded), the curve $C'$ must be a curve of best $D$-approximation to $P$, once we show that there is no curve in the locus $B$ with a smaller $\alpha$-value than $C'$.

To handle curves in the locus $B$, first note that the self-intersection $(2\pi^*D-39E)^2=-65$ is negative. Now, $B$ is contained in the locus defined by the vanishing $g^{13}$, namely the divisor $13C''$. Thus, it suffices to show that $\alpha_{P,C''}(D)>\alpha_{P,C'}(D)$. This is the case since $C''$ has degree $56$, so by Theorem \ref{thm:curve}, we see $\alpha_{P,C''}(D)\geq \frac{56}{2}=28>19.5=\alpha_{P,C'}(D)$. Therefore, $C'$ is indeed a curve of best approximation to $P$, provided that $\sqrt{-3}\in k_v\setminus k$.

In summary, the curve of best approximation depends in a subtle way on the number field $k$. In particular, if one wishes to show the existence of a curve of best $D$-approximation to $P$ without assuming \emph{a priori} that $P$ is canonically bounded, one would need to provide an explanation for the non-trivial fact that $C'$ is contained in the Zariski closed locus of exceptions to the canonical boundedness condition provided by Vojta's Conjecture, at least when $\sqrt{-3}\in k_v\setminus k$.

\end{document}